\documentclass[copyright,bookmarks,svgnames]{EPTCS/eptcs} 

\usepackage{underscore}
\usepackage{preamble/preamble}

\title{Universal Properties of Lens Proxy Pullbacks}
\def\titlerunning{Universal Properties of Lens Proxy Pullbacks}

\author{Matthew Di Meglio
\institute{Laboratory for Foundations of Computer Science\\
School of Informatics\\
University of Edinburgh\\
Edinburgh, Scotland}
\email{m.dimeglio@ed.ac.uk}
}
\def\authorrunning{Matthew Di Meglio}

\providecommand{\event}{ACT 2022}

\begin{document}
\maketitle

\begin{abstract}
A comprehensive account of the categorical properties of the category of small categories and asymmetric delta lenses is given in the recent works of Chollet et al.\ and Di Meglio. An important construction for proving many of these properties is Johnson and Rosebrugh’s \textit{``pullback''} of lenses, which we call the \textit{proxy pullback} of lenses. We give a new treatment of the proxy pullback in terms of \textit{compatibility}—a stronger notion of commutativity for squares of lenses. The proxy pullback is sometimes, but not always, a real pullback. Using new notions of \textit{sync-minimal} and \textit{independent} lens spans, we characterise when a lens span that forms a commuting square with a lens cospan has a comparison lens to a proxy pullback of the cospan.
\end{abstract}

\section{Introduction}

A \textit{bidirectional transformation} is a specification of when the joint state of two systems should be regarded as consistent, together with a protocol for updating each system to restore consistency in response to a change in the other~\cite{gibbons:2018:bidirectionaltransformations}. An \textit{asymmetric} bidirectional transformation is one where the state of one of the systems, called the \textit{view}, is completely determined by that of the other, called the \textit{source}.

A \textit{symmetric delta lens} is a mathematical model of a bidirectional transformation in which both of the systems involved are modelled as categories of states and transitions (deltas) rather than merely as sets of states, and the consistency restoration operations are aware of specifically which transition occurred rather than merely the state resulting from it~\cite{diskin:2011:statetodeltabxsym}. An \textit{asymmetric delta lens} is, in a similar way, a mathematical model of an asymmetric bidirectional transformation~\cite{diskin:2011:statetodeltabx}. Johnson and Rosebrugh established an equivalence between the symmetric delta lenses between two categories and the spans of asymmetric delta lenses between the categories modulo a certain equivalence relation~\cite{JohnsonRosebrugh:2015:SpansDeltaLenses}. This correspondence is important because, although we usually need the level of generality afforded by symmetric delta lenses to describe real-world bidirectional transformations, we would rather work with asymmetric delta lenses as they are easier to reason about.

Bidirectional transformations can be chained together; this is modelled mathematically by composition of symmetric delta lenses. Under the equivalence described above, composition of spans of asymmetric delta lenses is achieved not by pullbacks, which do not even always exist~\cite{DiMeglio:2021:Thesis}, but by a seemingly ad hoc pullback-like construction that Johnson and Rosebrugh called the \textit{``pullback''} (with quotation marks)~\cite{JohnsonRosebrugh:2015:SpansDeltaLenses}. As in prior work~\cite{DiMeglio:2021:CoequalisersUnderTheLens}, we adopt the name \textit{proxy pullback} from Bumpus and Kocsis~\cite{bumpus:2021:spined-categories:-generalizing-tree-width}. 

Our goal in this work is to understand in what sense the notion of lens proxy pullback is actually canonical. In category theory, canonicity is usually formalised by a universal property. Unfortunately, the obvious universal property, the one that characterises real pullbacks, does not always hold for lens proxy pullbacks. However, \(\Lens\) is not the only category with proxies for real pullbacks. Others include
\begin{itemize}[nosep]
\item the category of Polish probability spaces and measure preserving maps~\cite{simpson:2018:category-theoretic-structure-for-independence},
\item the category of smooth manifolds and smooth maps~\cite{Yassine:2020:GeneralizedSpanCategories}, and
\item the category of comonoids in a symmetric monoidal category with nice coreflexive equalisers~\cite{bohm:2019:crossed-modules-monoids-relative}.
\end{itemize}
Several frameworks for understanding such pullback-like constructions have been proposed, including Simpson's local independent products~\cite{simpson:2018:category-theoretic-structure-for-independence}, Böhm's relative pullbacks~\cite{bohm:2019:crossed-modules-monoids-relative} and Yassine's \(F\)-pullbacks~\cite{Yassine:2020:GeneralizedSpanCategories}.  In particular, Simpson and Böhm's approaches are both based on the idea that although a proxy pullback of a cospan may not be universal amongst \textit{all} spans that form a commuting square with the cospan, it should be universal amongst \textit{some class} of those spans. Inspired by this idea, we answer the main question
\begin{quoting}[font={itshape, center}]
    Amongst which lens spans is a lens proxy pullback universal?
\end{quoting}

Although the notion of lens proxy pullback is fundamentally about modelling composites of bidirectional transformations, it has recently also become an important tool for understanding the theory of lenses. In the comprehensive account of the categorical properties of the category \(\Lens\) of small categories and lenses given by Chollet et al.~\cite{Clarke:2021:CategoryLens} and Di Meglio~\cite{DiMeglio:2021:CoequalisersUnderTheLens}, the proxy pullback played the role of a real pullback in many of the proofs. With an answer to our main question, we address some of the open questions posed by Chollet et al.~\cite{Clarke:2020:InternalLensesAsFunctorsAndCofunctors} about real pullbacks in \(\Lens\) and their relationship with proxy pullbacks.

\subsection*{Outline}
In \cref{Section: Compatible Square}, we reformulate the notion of proxy pullback in terms of \textit{compatibility}—a commutativity-like property of squares of (asymmetric delta) lenses. Our new definition is better suited to category theory than the original one given by Johnson and Rosebrugh.

To answer our main question, our overall approach is to search for lens span properties that are possessed by proxy-pullback spans and that are preserved by precomposition with lenses, until we find enough of them that a lens span posessing all of these conditions is guaranteed to have a comparison lens to the proxy pullback. Only two such properties are needed: the first is compatibility with the cospan; the second is a new property of lens spans that we call \textit{independence}, which is itself defined in terms of another new property of lens spans that we call \textit{sync minimality}. We introduce these properties in \cref{Section: Sync-minimal spans}, and we prove their necessity in \cref{Section: necessity}. In \cref{Sync minimal proxy pullback terminal amongst compatible independent}, one of the main results of this paper, we see that for sync-minimal proxy-pullback spans, the possession of these two properties by a lens span is also sufficient for the existence of such a comparison lens to the proxy-pullback span.

A natural next step is then to determine whether the sync minimality of a proxy-pullback span is itself also a necessary condition for the existence of such a comparison lens. This is not in general true, however, in \cref{Sync minimal necessary for terminality}, we see that it is actually a necessary condition for the \textit{simultaneous} existence of a comparison lens to the proxy pullback from all independent lens spans that are compatible with the cospan. Stated differently, if a proxy-pullback span of a lens cospan is terminal amongst the independent spans that are compatible with the cospan, then the proxy-pullback span is necessarily sync minimal.

Combining the above results, a proxy pullback is a real pullback if and only if it is sync minimal and every lens span that forms a commuting square with the cospan is compatible with the cospan and is also independent. Although this statement completely characterises when a proxy pullback is a real pullback, it is somewhat unsatisfactory, as it is not expressed in terms of properties of the cospan that are easily checked. There is, however, a tractable characterisation for lens cospans whose apex is the terminal category. Indeed, a proxy product of two categories is a real product if and only if at least one of the two categories is a discrete category. This result, \cref{Proxy product of discrete categories}, is the pinnacle of \cref{Section: proxy pullbacks of split opfibrations,Section: proxy pullbacks of discrete opfibrations and proxy products}. Finding such a tractable characterisation for general lens cospans is ongoing work.

\begin{remark}
This paper is based on Chapter~4 of the author's Master of Research thesis~\cite{DiMeglio:2021:Thesis}.
\end{remark}

\section{Background}
\label{Section: Background}

\subsection{Notation}
\label{Section: Notation}

Application of functions (functors, etc.) is written by juxtaposing the function name with its argument, and parentheses are only used when needed. Binary operators like \(\compose\) have lower precedence than application, so an expression like \(F a \compose F b\) parses as \((F a) \compose (F b)\).

Let \(\Cat\) denote the category whose objects are small categories and whose morphisms are functors. Categories with boldface names \(\A\), \(\B\), \(\C\), etc.\ are always small. We write \(\objectSet{\C}\) for the set of objects of a small category \(\C\), and, for all \(X, Y \in \objectSet{\C}\), we write \(\homSet{\C}{X}{Y}\) for the set of morphisms of \(\C\) from \(X\) to \(Y\). For each \(X \in \objectSet{\C}\), we write \(\outSet{\C}{X}\) for the set \(\DisjointUnion_{Y \in \objectSet{\C}}\homSet{\C}{X}{Y}\) of all morphisms in \(\C\) out of \(X\). We write \(\source{f}\) and \(\target{f}\) for, respectively, the source and target of a morphism \(f\). We also write \(f \colon X \to Y\) to say that \(X, Y \in \objectSet{\C}\) and \(f \in \homSet{\C}{X}{Y}\). The composite of morphisms \(f \colon X \to Y\) and \(g \colon Y \to Z\) is denoted \(g \compose f\).

The category with a single object \(0\) and no non-identity morphisms, also known as the \textit{terminal category}, is denoted \(\terminalCat\). The category with two objects \(0\) and \(1\) and a single non-identity morphism, namely \(u \colon 0 \to 1\), also known as the \textit{interval category}, is denoted \(\intervalCat\). We will identify objects and morphisms of a small category \(\C\) with the corresponding functors \(\terminalCat \to \C\) and \(\intervalCat \to \C\) respectively.

If the square in \(\Cat\)
\begin{equation*}
    \label{Equation: Pullback square}
    \begin{tikzcd}
        \D \arrow[r, "T"]\arrow[d, "S" swap] & \B \arrow[d, "G"]\\
        \A \arrow[r, "F" swap] & \C
    \end{tikzcd}
\end{equation*}
is a pullback square and \(S' \colon \D' \to \A\) and \(T' \colon \D' \to \B\) are functors for which \(F \compose S' = G \compose T'\), then we write \(\pair{S'}{T'}\) for the functor \(\D' \to \D\) induced from \(S'\) and \(T'\) by the universal property of the pullback. By our above identification of objects with functors from~\(\terminalCat\), if \(A \in \objectSet{\A}\) and \(B \in \objectSet{\B}\) are such that \(FA = GB\), then \(\pair{A}{B}\) is the object of \(\D\) selected by the functor \(\terminalCat \to \D\) induced by the universal property of the pullback from the functors \(\terminalCat \to \A\) and \(\terminalCat \to \B\) that respectively select the objects \(A\) and \(B\).

\subsection{Cofunctors and Lenses}
\label{Section: Cofunctors and lenses}

The definition of (asymmetric delta) lens most useful to us will be as a suitable pairing of a functor and a cofunctor~\cite{AhmanUustalu:2017:TakingUpdatesSeriously}. Let us first recall the definition of a cofunctor~\cite{Aguiar:1997:InternalCategoriesAndQuantumGroups, Clarke:2020:InternalLensesAsFunctorsAndCofunctors}.

\begin{definition}
For small categories \(\A\) and \(\B\), a \textit{cofunctor} \(F \colon \A \to \B\) consists of
\begin{itemize}
    \item a function \(F \colon \objectSet{\A} \to \objectSet{\B}\), called the \textit{object function}, and
    \item functions \(\lift{F}{A} \colon \outSet{\B}{FA} \to \outSet{\A}{A}\) for all \(A \in \objectSet{\A}\), called \textit{lifting functions},
\end{itemize}
such that the equations
\[
\begin{array}{ccccc}
    F\target \lift{F}{A}b = \target b &\qquad&
    \lift{F}{A}\id{FA} = \id{A}&\qquad&
    \lift{F}{A}(b' \compose b) = \lift{F}{A'}b' \compose \lift{F}{A}b\\
    \text{(PutTgt)}&&
    \text{(\PutId{})}&&
    \text{(\PutPut{})}
\end{array}
\]
hold whenever they are defined.
\end{definition}

\begin{warning}
The notions of cofunctor and contravariant functor are distinct and unrelated.
\end{warning}

There is a category \(\Cof\) whose objects are small categories and whose morphisms are cofunctors. The composite \(G \compose F\) of cofunctors \(F \colon \A \to \B\) and \(G \colon \B \to \C\) has as its object function the composite of the object functions of \(F\) and \(G\), and has \(\lift{(G \compose F)}{A}c = \lift{F}{A}\lift{G}{FA}c\) for all \(A \in \objectSet{\A}\) and all \(c \in \outSet{\C}{GFA}\).

In the following definition of a lens, although we use the name of the lens to refer both to its get functor and its put cofunctor, the equal object function requirement ensures that there is no ambiguity.

\begin{definition}
For small categories \(\A\) and \(\B\), a \textit{lens} \(F \colon \A \to \B\) consists of 
\begin{itemize}
    \item a functor \(F \colon \A \to \B\), called the \textit{get functor}, and
    \item a cofunctor \(F \colon \A \to \B\), called the \textit{put cofunctor},
\end{itemize}
with same object functions, such that the equation
\[
F\lift{F}{A}b = b \tag{\PutGet{}}
\]
holds whenever it is defined.
\end{definition}

There is a category \(\Lens\) of small categories and lenses. There are also identity-on-objects functors
\(\forgetLensToCat \colon \Lens \to \Cat\) and \(\forgetLensToCof \colon \Lens \to \Cof\) that respectively send a lens to its get functor and put cofunctor.

\subsection{Discrete Opfibrations and Split Opfibrations}
\label{Section: Discrete opfibrations and split opfibrations}

\begin{definition}
A functor \(F \colon \A \to \B\) is a \textit{discrete opfibration} if, for each \(A \in \objectSet{\A}\) and each \(b \in \outSet{\B}{FA}\), there is a unique \(a \in \outSet{\A}{A}\) such that \(Fa = b\).
\end{definition}

\begin{definition}
A lens \(F \colon \A \to \B\) is a \textit{discrete opfibration} if the equation
\[
\lift{F}{A}Fa = a \tag{GetPut}
\]
holds for each \(A \in \objectSet{\A}\) and each \(a \in \outSet{\A}{A}\).
\end{definition}

\begin{warning}
The name \GetPut{} has, in the past, been used for what is now called \PutId{}. The reader should note that lenses in general need not satisfy \GetPut{} the way that we have defined it.
\end{warning}

If \(F \colon \A \to \B\) is a discrete opfibration, then there is a unique lens mapped by \(\forget\) to \(F\), which we sometimes also refer to as \(F\). A lens is a discrete opfibration if and only if its get functor is a discrete opfibration. Together, these results mean that we need not specify whether a discrete opfibration \(F \colon \A \to \B\) is a functor or a lens, and we can use the name \(F\) in both functor and lens contexts without ambiguity.

\begin{definition}
For a functor \(F \colon \A \to \B\), a morphism \(f \colon X \to Y\) in \(\A\) is \textit{\(F\)-opcartesian} if, for all morphisms \(f' \colon X \to Y'\) in \(\A\) and all morphisms \(v \colon FY \to FY'\) in \(\B\) such that \(Ff' = v \compose Ff\), there is a unique morphism \(u \colon Y \to Y'\) in \(\A\) such that \(f' = u \compose f\) and \(v = Fu\). For \(f\) to be \textit{weakly \(F\)-opcartesian}, the property described in the previous sentence need only hold for \(v = \id{FY}\).
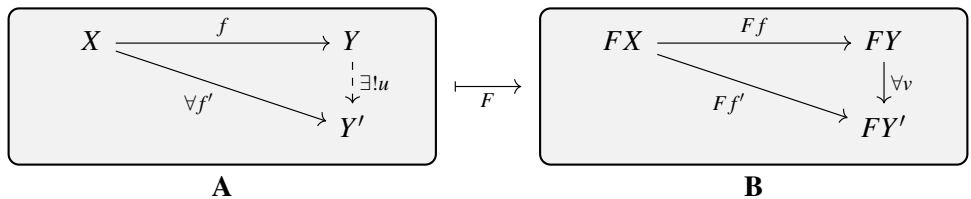
\[
\begin{tikzpicture}[
    x=3.6em, y=2.0em,
    commutative diagrams/.cd,
    every diagram
]

    \begin{scope}[shift={(0,0)}]
        \node[boxed, fit={(-2.05,-1.35) (2.05,1.35)}, label={below:\(\A\)}] (Acat) {};
    
        \begin{scope}[
            commutative diagrams/.cd,
            every node,
            every cell
        ]
            \node (UL)   at (-1.25, 0.75) {\(X\)};
            \node (UR)   at ( 1.25, 0.75) {\(Y\)};
            \node (DR)   at ( 1.25, -0.75) {\(Y'\)};
        \end{scope}
        
        \begin{scope}[
            commutative diagrams/.cd,
            every arrow,
            every label
        ]
            \draw (UL) -- node {\(f\)} (UR);
            \draw (UL) -- node[swap] {\(\forall f'\)} (DR);
            \draw[dashed] (UR) -- node {\(\exists! u\)} (DR);
        \end{scope}
    \end{scope}
    
    \begin{scope}[shift={(5.1,0)}]
        \node[boxed, fit={(-2.05,-1.35) (2.05,1.35)}, label={below:\(\B\)}] (Bcat) {};
    
        \begin{scope}[
            commutative diagrams/.cd,
            every node,
            every cell
        ]
            \node (UL)   at (-1.25, 0.75) {\(FX\)};
            \node (UR)   at ( 1.25, 0.75) {\(FY\)};
            \node (DR)   at ( 1.25, -0.75) {\(FY'\)};
        \end{scope}
        
        \begin{scope}[
            commutative diagrams/.cd,
            every arrow,
            every label
        ]
            \draw (UL) -- node {\(Ff\)} (UR);
            \draw (UL) -- node[swap] {\(Ff'\)} (DR);
            \draw (UR) -- node {\(\forall v\)} (DR);
        \end{scope}
    \end{scope}
  
    \begin{scope}[
        commutative diagrams/.cd,
        every arrow,
        every label,
        shorten=0.2em,
        mapsto
    ]
        \draw (Acat) -- node[swap] {\(F\)} (Bcat);
    \end{scope}
\end{tikzpicture}
\]
\end{definition}

\begin{definition}
A lens \(F \colon \A \to \B\) is a \textit{split opfibration} if each morphism \(\lift{F}{A}b\) is \(\forgetLensToCat F\)-opcartesian.
\end{definition}

\begin{proposition}
A lens \(F \colon \A \to \B\) is a split opfibration if and only if, for all \(a \colon A \to A'\) in~\(\A\), there is a unique \(u \colon \target \lift{F}{A}Fa \to A'\) in \(\A\) such that \(a = u \compose \lift{F}{A}Fa\) and \(Fu = \id{FA'}\).
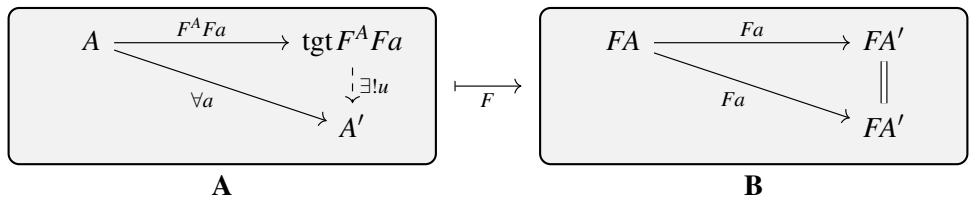
\[
\begin{tikzpicture}[
    x=3.6em, y=2.0em,
    commutative diagrams/.cd,
    every diagram
]

    \begin{scope}[shift={(0,0)}]
        \node[boxed, fit={(-2.05,-1.35) (2.05,1.35)}, label={below:\(\A\)}] (Acat) {};
    
        \begin{scope}[
            commutative diagrams/.cd,
            every node,
            every cell
        ]
            \node (UL)   at (-1.25, 0.75) {\(A\)};
            \node (UR)   at ( 1.25, 0.75) {\(\target \lift{F}{A}Fa\)};
            \node (DR)   at ( 1.25, -0.75) {\(A'\)};
        \end{scope}
        
        \begin{scope}[
            commutative diagrams/.cd,
            every arrow,
            every label
        ]
            \draw (UL) -- node {\(\lift{F}{A}Fa\)} (UR);
            \draw (UL) -- node[swap] {\(\forall a\)} (DR);
            \draw[dashed] (UR) -- node {\(\exists! u\)} (DR);
        \end{scope}
    \end{scope}
    
    \begin{scope}[shift={(5.1,0)}]
        \node[boxed, fit={(-2.05,-1.35) (2.05,1.35)}, label={below:\(\B\)}] (Bcat) {};
    
        \begin{scope}[
            commutative diagrams/.cd,
            every node,
            every cell
        ]
            \node (UL)   at (-1.25, 0.75) {\(FA\)};
            \node (UR)   at ( 1.25, 0.75) {\(FA'\)};
            \node (DR)   at ( 1.25, -0.75) {\(FA'\)};
        \end{scope}
        
        \begin{scope}[
            commutative diagrams/.cd,
            every arrow,
            every label
        ]
            \draw (UL) -- node {\(Fa\)} (UR);
            \draw (UL) -- node[swap] {\(Fa\)} (DR);
            \draw[commutative diagrams/equals] (UR) -- (DR);
        \end{scope}
    \end{scope}
  
    \begin{scope}[
        commutative diagrams/.cd,
        every arrow,
        every label,
        shorten=0.2em,
        mapsto
    ]
        \draw (Acat) -- node[swap] {\(F\)} (Bcat);
    \end{scope}
\end{tikzpicture}
\]
\end{proposition}

\begin{proof}[Proof sketch.]
Having opcartesian lifts is equivalent to having weakly opcartesian lifts that are closed under composition. The chosen lifts of a \textit{lens} are, by the \PutPut{} axiom, closed under composition.
\end{proof}

In particular, every discrete opfibration is a split opfibration.

\section{Compatible Squares and Proxy Pullbacks}
\label{Section: Compatible Square}

Compatibility is a stronger notion of commutativity for a square of lenses. In addition to requiring that the underlying square of functors and the underlying square of cofunctors commute, it also imposes conditions on certain squares formed from a mix of the underlying functors and the underlying cofunctors.



\begin{definition}
A \textit{compatible lens square} is a commuting lens square
\begin{equation}
\label{Equation: square of lenses}
\begin{tikzcd}
\D  \arrow[d, "\Gbar" swap]
    \arrow[r, "\Fbar"]&
\B  \arrow[d, "G"]\\
\A  \arrow[r, "F" swap]&
\C
\end{tikzcd}
\end{equation}
such that the compatibility equations
\begin{align*}
    \Fbar \lift{\Gbar}{D}a = \lift{G}{\Fbar D} F a
    &&
    \Gbar \lift{\Fbar}{D}b = \lift{F}{\Gbar D} G b
\end{align*}
hold whenever they are defined. We also say that \((\Gbar, \Fbar)\) is \textit{compatible with} \((F,G)\).
\end{definition}

\begin{proposition}
\label{Compatible square and discrete opfibration}
Every commuting lens square for which one leg of the cospan is a discrete opfibration is a compatible lens square.
\end{proposition}

\begin{proof}
For all \(D \in \objectSet{\D}\), all \(a \in \outSet{\A}{\Gbar D}\) and all \(b \in \outSet{\B}{\Fbar D}\), we have
\begin{gather*}
    \Fbar \lift{\Gbar}{D}a = \Fbar \lift{\Gbar}{D} \lift{F}{\Gbar D}Fa = \Fbar \lift{(\Gbar \compose F)}{D}Fa = \Fbar \lift{(\Fbar \compose G)}{D} Fa = \Fbar \lift{\Fbar}{D}\lift{\Gbar}{\Fbar D}Fa = \lift{\Gbar}{\Fbar D}Fa,\\
    \Gbar \lift{\Fbar}{D}b = \lift{F}{\Gbar D} F \Gbar \lift{\Fbar}{D}b = \lift{F}{\Gbar D}G\Fbar\lift{\Fbar}{D}b = \lift{F}{\Gbar D}Gb.
    \qedhere
\end{gather*}
\end{proof}

\begin{remark}
\label{Commuting lens triangle is compatible lens square}
As identity lenses are discrete opfibrations, every commuting lens triangle becomes a compatible lens square by inserting an identity lens into the triangle in the appropriate place.
\end{remark}

\begin{definition} A \textit{proxy-pullback square} is a compatible lens square sent by \(\forgetLensToCat\) to a pullback square. A \textit{proxy pullback} of a lens cospan is a lens span forming a proxy-pullback square with the cospan. A \textit{proxy product} is a proxy pullback of a cospan whose apex is the terminal category.
\end{definition}

In diagrams, we will mark proxy-pullback squares with \(\mathsf{PPB}\).

Suppose that the lens square \cref{Equation: square of lenses} is mapped by \(\forgetLensToCat\) to a pullback square. By the universal property of this pullback square, the compatibility conditions for \cref{Equation: square of lenses} to be a proxy-pullback square are equivalent to the equations
\(\lift{\Gbar}{D}a = \pair{a}{\lift{G}{\Fbar D}Fa}\) and \(\lift{\Fbar}{D}b = \pair{\lift{F}{\Gbar D} Gb}{b}\). Actually, starting with a pullback in \(\Cat\) of the get functors of a lens span, these equations define lifts on the pullback projection functors; one may check that this turns these functors into lenses and that the resulting lens square is compatible.

\begin{proposition}
\label{Proxy pullback creation}
For each lens cospan, there is a unique proxy pullback of the cospan above each pullback of the get functors of the cospan.
\end{proposition}

\begin{proposition}[{\cite[Corollary~3.15]{DiMeglio:2021:Thesis}}]
\label{Proxy pullback unique up to iso}
Proxy-pullback spans are unique up to unique span isomorphism.
\end{proposition}

\section{Sync-minimal and Independent Lens Spans}
\label{Section: Sync-minimal spans}
\label{Section: Independent spans}

As was explained in the introduction to this paper, new notions of \textit{sync minimality} and \textit{independence} of lens spans give various necessary and sufficient conditions for the existence of comparison lenses into proxy pullbacks. In this section, we merely introduce these notions, delaying the development of their theory to when it is needed later in the paper.

Johnson and Rosebrugh~\cite{JohnsonRosebrugh:2015:SpansDeltaLenses} proposed that we regard a lens span \(\A \xleftarrow{F} \C \xrightarrow{G} \B\) as a synchronisation protocol between the systems represented by the categories \(\A\) and \(\B\). From this perspective, the category~\(\C\) has the sole purpose of coordinating the propagation to \(\B\) of transitions that occur in \(\A\) and vice versa. As transitions always originate in \(\A\) or~\(\B\), there may be morphisms in \(\C\) that are never used—these are the ones that are not composites of a sequence of morphisms that are all lifts along \(F\) or \(G\). If there are no such extraneous morphisms in \(\C\), we call the lens span \textit{sync minimal}.

\begin{definition}
A lens span
\[\begin{tikzcd}
    \A
&   \C
    \arrow[l, "F" swap]
    \arrow[r, "G"]
&   \B
\end{tikzcd}\]
is \textit{sync minimal} if each morphism in \(\C\) is a composite of a sequence of morphisms
\[
\begin{tikzcd}[column sep=large]
C_1 \arrow[r, "c_1"]&
C_2 \arrow[r, "c_2"]&
C_3 \arrow[r, phantom, "\cdots"]&[-2em]
C_{n-1} \arrow[r, "c_{n-1}"]&
C_{n}
\end{tikzcd}
\]
that are all lifts along \(F\) or \(G\), that is, for each \(k\), either \(c_k = \lift{F}{C_k} F c_k\) or \(c_k = \lift{G}{C_k}Gc_k\).
\end{definition}

There are many sync-minimal lens spans, but not all proxy-pullback spans are sync minimal.

\begin{example}
\label{Example: Non sync minimal proxy pullback}
Consider the proxy-pullback square depicted in the diagram below, where the lens lifts are indicated by the colouring of the morphisms. 
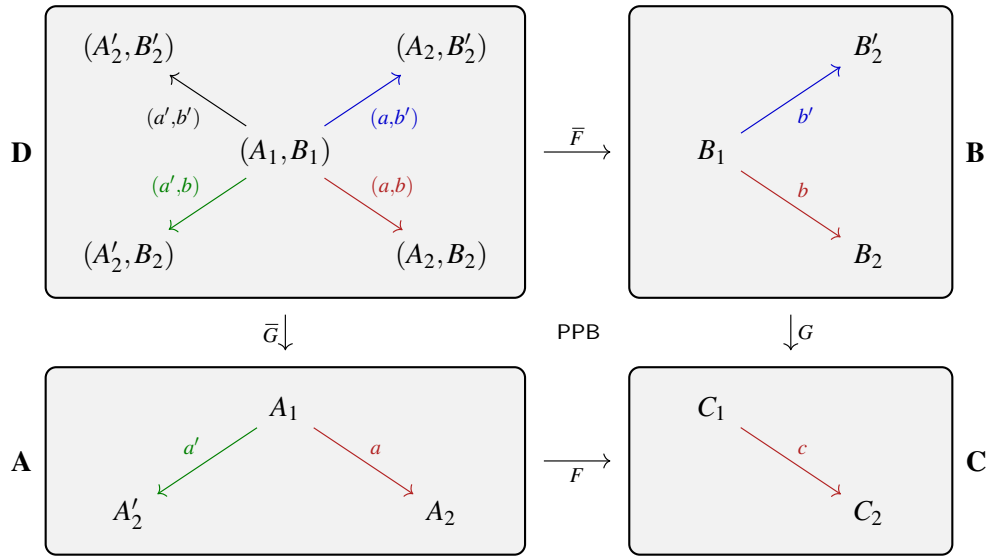
\[
\begin{tikzpicture}[
    x=3.6em, y=2.4em,
    commutative diagrams/.cd,
    every diagram
]
    
    \begin{scope}[shift={(0,0)}]
        \node[boxed, fit={(-2.3,-2.1) (2.3, 2.1)}, label={left:\(\D\)}] (Dcat) {};
    
        \begin{scope}[
            commutative diagrams/.cd,
            every node,
            every cell
        ]
            \node (C)  at ( 0, 0) {\((A_1,B_1)\)};
            \node (UL) at (-1.5, 1.5) {\((A_2', B_2')\)};
            \node (UR) at ( 1.5, 1.5) {\((A_2, B_2')\)};
            \node (DL) at (-1.5,-1.5) {\((A_2', B_2)\)};
            \node (DR) at ( 1.5,-1.5) {\((A_2, B_2)\)};
        \end{scope}
        
        \begin{scope}[
            commutative diagrams/.cd,
            every arrow,
            every label
        ]
            \draw[FireBrick] (C) edge node {\((a, b)\)}  (DR);
            \draw[Green, swap] (C) -- node {\((a', b)\)}  (DL);
            \draw[MediumBlue] (C) -- node[swap] {\((a, b')\)}  (UR);
            \draw (C) -- node {\((a', b')\)}  (UL);
        \end{scope}
    \end{scope}
    
    \begin{scope}[shift={(0,-3.7)}]
        \node[boxed, fit={(-2.3,-2.1) (2.3, 0.6)}, label={left:\(\A\)}] (Acat) {};
    
        \begin{scope}[
            commutative diagrams/.cd,
            every node,
            every cell
        ]
            \node (C) at ( 0  , 0  ) {\(A_1\)};
            \node (L) at (-1.5,-1.5) {\(A_2'\)};
            \node (R) at ( 1.5,-1.5) {\(A_2\)};
        \end{scope}
        
        \begin{scope}[
            commutative diagrams/.cd,
            every arrow,
            every label
        ]
            \draw[Green] (C) -- node[swap] {\(a'\)} (L);
            \draw[FireBrick] (C) -- node {\(a\)} (R);
        \end{scope}
    \end{scope}
    
    \begin{scope}[shift={(4.1,0)}]
        \node[boxed, fit={(-0.8,-2.1) (2.3,2.1)}, label={right:\(\B\)}] (Bcat) {};
    
        \begin{scope}[
            commutative diagrams/.cd,
            every node,
            every cell
        ]
            \node (C) at ( 0  , 0  ) {\(B_1\)};
            \node (U) at ( 1.5, 1.5) {\(B_2'\)};
            \node (D) at ( 1.5,-1.5) {\(B_2\)};
        \end{scope}
        
        \begin{scope}[
            commutative diagrams/.cd,
            every arrow,
            every label
        ]
            \draw[FireBrick] (C) -- node {\(b\)} (D);
            \draw[MediumBlue] (C) -- node[swap] {\(b'\)} (U);
        \end{scope}
    \end{scope}
    
    \begin{scope}[shift={(4.1,-3.7)}]
        \node[boxed, fit={(-0.8,-2.1) (2.3,0.6)}, label={right:\(\C\)}] (Ccat) {};
    
        \begin{scope}[
            commutative diagrams/.cd,
            every node,
            every cell
        ]
            \node (L) at (0, 0) {\(C_1\)};
            \node (R) at (1.5,-1.5) {\(C_2\)};
        \end{scope}
        
        \begin{scope}[
            commutative diagrams/.cd,
            every arrow,
            every label
        ]
            \draw[FireBrick] (L) -- node {\(c\)} (R);
        \end{scope}
    \end{scope}
  
    \begin{scope}[
        commutative diagrams/.cd,
        every arrow,
        every label,
        shorten=0.2em
    ]
        \draw (Acat) -- node[swap]  {\(F\)} (Ccat);
        \draw (Bcat) -- node        {\(G\)} (Ccat);
        \draw (Dcat) -- node[swap]  {\(\Gbar\)} (Acat);
        \draw (Dcat) -- node        {\(\Fbar\)} (Bcat);
    \end{scope}
  
    \draw (Dcat) edge[commutative diagrams/.cd, phantom] node[font=\scriptsize] {\(\PPB\)} (Ccat);
\end{tikzpicture}
\]
The lens span \((\Gbar, \Fbar)\) is not sync minimal as the morphism \((a',b')\) is not a composite of lifts. Notice that removing \((a',b')\) from \(\D\) would make the span \((\Gbar,\Fbar)\) sync minimal.
\end{example}

Starting with a lens span \(\A \xleftarrow{F} \C \xrightarrow{G} \B\), by removing all morphisms in \(\C\) that are not composites of a sequence of morphisms that are lifts along \(F\) or \(G\), we obtain a sync-minimal lens span from \(\A\) to \(\B\) that encodes the same synchronisation protocol as \((F,G)\). We call this sync-minimal lens span the \textit{sync-minimal core} of \((F,G)\) and denote it by \(\sync (F,G)\). Let \(\syncCounit{(F,G)}\) denote the inclusion functor from the apex of \(\sync (F,G)\) to \(\C\).

We are now ready to define the notion of \textit{independence} for lens spans. It is similar to a jointly-monic condition, except only with respect to morphisms in the apex of the sync-minimal core of the span with the same source object. Defining independence with respect to the sync-minimal core is necessary for independence to be preserved by precomposition with lenses.

\begin{definition}
A lens span \(\A \xleftarrow{F} \C \xrightarrow{G} \B\) is called \textit{independent} if, for all morphisms \(c\) and \(c'\) in the apex of \(\sync(F,G)\) with the same source, whenever \(Fc = Fc'\) and \(Gc = Gc'\), also \(c = c'\).
\end{definition}

\begin{remark}
Simpson~\cite{simpson:2018:category-theoretic-structure-for-independence} defines the notion of \textit{independent product} with respect to a chosen \textit{independence structure}—a multicategory of multispans, called \textit{independent} multispans, that satisfies certain additional properties. This is where our terminology for independent lens spans originates. We will have more to say about Simpson's independent products and local independent products at the end of this paper.
\end{remark}

The lens span \((\Gbar, \Fbar)\) in \cref{Example: Non sync minimal proxy pullback} is independent.

\section{Necessity of Compatibility and Independence}
\label{Section: necessity}

Proxy-pullback spans of a lens cospan are, by definition, compatible with the cospan. In this section, we will show that proxy-pullback spans are also independent, and that compatibility and independence of lens spans are preserved by precomposition with lenses. It follows that whenever a lens span that commutes with a lens cospan has a comparison lens to the proxy pullback of the cospan, the span is necessarily independent and compatible with the cospan.

\begin{proposition}
\label{Proxy pullback is independent}
All proxy-pullback spans are independent.
\end{proposition}

\begin{proof}
Let \(\A \xleftarrow{\Gbar} \D \xrightarrow{\Fbar} \B\) be a proxy pullback of some lens cospan. For all \(D \in \objectSet{\D}\), and all \(d,d' \in \outSet{\D}{D}\), if \(\Fbar d = \Fbar d'\) and \(\Gbar d = \Gbar d'\) then \(d = d'\) by the universal property of the pullback in \(\Cat\) underlying the proxy pullback. In particular, this holds for those objects and morphisms in the apex of \(\sync{(\Gbar,\Fbar)}\).
\end{proof}

\begin{proposition}
\label{Compatibility preserved by precomposition}
Consider the following diagram in \(\Lens\), where \(K_1 = K_2 \compose H\) and \(J_1 = J_2 \compose H\).
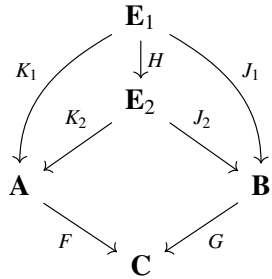
\[\begin{tikzcd}[row sep=scriptsize]
&\E_1 \arrow[d, "H"]\arrow[ddl, "K_1" swap, out=-150, in=90]\arrow[ddr, "J_1", out=-30, in=90]&\\
&\E_2\arrow[dl, "K_2"{swap, near start}] \arrow[dr, "J_2" near start]&\\
\A \arrow[dr, "F" swap] && \B \arrow[dl, "G"]\\
&\C&
\end{tikzcd}\]
If the span \((K_2,J_2)\) is compatible with the cospan \((F,G)\), then so is the span \((K_1,J_1)\).
\end{proposition}

\begin{proof}
Suppose that the span \((K_2,J_2)\) is compatible with the cospan \((F,G)\). Then the span \((K_1,J_1)\) forms a commuting square with the cospan \((F,G)\). One of the compatibility conditions holds because \(J_1\lift{K_1}{E}a = J_2H\lift{H}{E}\lift{K_2}{HE}a = J_2 \lift{K_2}{HE}a = \lift{G}{J_2HE}Fa = \lift{G}{J_1E}Fa\), and the other holds similarly.
\end{proof}

\begin{proposition}
\label{Independence and lens span morphisms}
Consider the following commuting diagram in \(\Lens\).
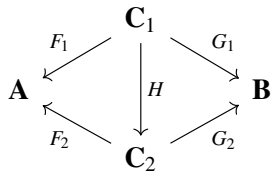
\[
\begin{tikzcd}[row sep=small]
&\C_1 \arrow[dl, "F_1" swap] \arrow[dr, "G_1"] \arrow[dd, "H"]&\\
\A&&\B\\
&\C_2 \arrow[ul, "F_2"] \arrow[ur, "G_2" swap]&
\end{tikzcd}
\]
If the span \((F_2, G_2)\) is independent, then the span \((F_1, G_1)\) is also independent.
\end{proposition}

\begin{proof}
Suppose that \(c\) and \(c'\) are morphisms in the apex of \(\sync (F_1,G_1)\) with the same source object \(C\) such that \(F_1c = F_1c'\) and \(G_1c = G_1c'\). Then \(Hc\) and \(Hc'\) are morphisms in the apex of \(\sync (F_2,G_2)\) with the same source object \(HC\) such that \(F_2Hc = F_2Hc'\) and \(G_2Hc = G_2Hc'\). As \((F_2, G_2)\) is independent, actually \(Hc = Hc'\). But \(c\) and \(c'\) are both composites of lifts along \(H \compose F_2\) and \(H \compose G_2\), so they are both lifts along \(H\), and thus \(c = \lift{H}{C}Hc = \lift{H}{C}Hc' = c'\).
\end{proof}

Although compatibility and independence give necessary conditions for the existence of a comparison lens, these conditions are not sufficient ones.

\begin{example}
Consider again the proxy-pullback square in \cref{Example: Non sync minimal proxy pullback}. The sync-minimal core \(\sync (\Gbar, \Fbar)\) of \((\Gbar, \Fbar)\) is obtained by removing the morphism \((a',b')\) from \(\D\). Although the span \(\sync (\Gbar, \Fbar)\) is independent and compatible with the cospan \((F,G)\), there is no comparison \textit{lens} from it to the proxy-pullback span \((\Gbar, \Fbar)\). Assume that such a comparison lens exists. Then, as the put cofunctor of the comparison lens commutes with the put cofunctors of the legs of both spans, all of the morphisms in the apex of \(\sync (\Gbar, \Fbar)\) are necessarily lifts of the corresponding morphisms in \(\D\). The \PutGet{} axiom necessitates that the lift by such a comparison lens of the morphism \((a',b')\) into the apex of \(\sync (\Gbar, \Fbar)\) be distinct from the lifts of the other morphisms \((a',b)\), \((a,b')\) and \((a,b)\), but there is no such morphism.
\end{example}

\section{Necessity and Sufficiency of Sync Minimality}
\label{Section: semiuniversal properties}

In the previous section we saw that a lens span that commutes with a lens cospan and has a comparison lens to the proxy pullback of the cospan is necessarily independent and compatible with the cospan. It turns out that if the proxy pullback is also sync minimal, then these conditions are also sufficient.

\begin{theorem}
\label{Sync minimal proxy pullback terminal amongst compatible independent}
Consider the following commuting diagram in \(\Lens\).
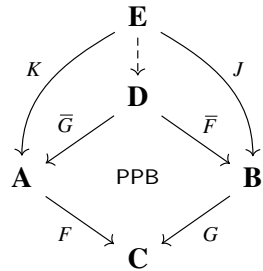
\[
\begin{tikzcd}[row sep=scriptsize]
&\E \arrow[ddl, "K" swap, out=-150, in=90]
    \arrow[ddr, "J", out=-30, in=90]
    \arrow[d, dashed]&\\
&\D \arrow[dl, "\Gbar" swap]
    \arrow[dr, "\Fbar"]
    \arrow[dd, PPB]&
    \\
\A  \arrow[dr, "F" swap]&&
\B  \arrow[dl, "G"]\\
&\C&
\end{tikzcd}
\]
Suppose that the span \((K,J)\) is independent and is compatible with the cospan \((F,G)\).
If the span \((\Gbar, \Fbar)\) is sync minimal, then there is a unique lens \(\E \to \D\) such that the triangles commute.
\end{theorem}

\begin{proof}
Suppose that \((\Gbar,\Fbar)\) is sync minimal. Let \(L \colon \E \to \D\) be the unique comparison functor from the span \((\forgetLensToCat K, \forgetLensToCat J)\) to the pullback span \((\forgetLensToCat \Gbar, \forgetLensToCat \Fbar)\). If there is a lens structure on \(L\) that makes the triangles commute, then, for all \(E \in \objectSet{\E}\), all \(a \in \outSet{\A}{KE}\) and all \(b \in \outSet{\B}{JE}\), we necessarily have
\begin{align*}
\lift{L}{E}\lift{\Gbar}{LE}a = \lift{K}{E}{a}&& \text{and}&& \lift{L}{E}\lift{\Fbar}{LE}b = \lift{J}{E}{b};
\end{align*}
that is, the lifts by \(L\) of those morphisms of \(\D\) that are lifts by \(\Gbar\) and \(\Fbar\) are determined by the lifts by \(K\) and~\(J\). As \((\Gbar, \Fbar)\) is sync minimal, each morphism in \(\D\) is a composite of lifts by \(\Gbar\) and lifts by \(\Fbar\), and so the above equations and the \PutPut{} axiom for \(L\) determine the lifts by \(L\) of all morphisms of~\(\D\). Such a lens structure on \(L\) is thus uniquely determined if it exists. 

In order to define \(\lift{L}{E}d\) using the above equations, we need to check, for any two decompositions
\[\lift{\Fbar}{D_n}b_n \compose \lift{\Gbar}{D_{n-1}}a_{n-1} \compose \cdots \compose \lift{\Fbar}{D_1}b_1 \compose \lift{\Gbar}{LE}a_0
\qquad\text{and}\qquad
\lift{\Fbar}{D_m'}b_m' \compose \lift{\Gbar}{D_{m-1}'}a_{m-1}'\cdots \compose \lift{\Fbar}{D_1'}b_1' \compose \lift{\Gbar}{LE}a_0'\]
of \(d\), that
\[\lift{J}{E_n}b_n \compose \lift{K}{E_{n-1}}a_{n-1} \compose \cdots \compose \lift{J}{E_1}b_1 \compose \lift{K}{E}a_0 =
\lift{J}{E_m'}b_m' \compose \lift{K}{E_{m-1}'}a_{m-1}'\cdots \compose \lift{J}{E_1'}b_1' \compose \lift{K}{E}a_0'.\]
Using the compatibility of both \((K,J)\) and \((\Gbar, \Fbar)\) with \((F,G)\), we see that
\begin{align*}
    \Fbar d
    &= \Fbar \paren[\big]{\lift{\Fbar}{D_n} b_n \compose \lift{\Gbar}{D_{n-1}}a_{n-1} \compose \cdots \compose \lift{\Fbar}{D_1} b_1 \compose \lift{\Gbar}{LE} a_0}\\
    &= \Fbar \lift{\Fbar}{D_n} b_n \compose \Fbar \lift{\Gbar}{D_{n-1}}a_{n-1} \compose \cdots \compose \Fbar \lift{\Fbar}{D_1} b_1 \compose \Fbar \lift{\Gbar}{LE} F a_0\\
    &= b_n \compose \lift{G}{\Fbar D_{n-1}} F a_{n-1} \compose \cdots \compose b_1 \compose \lift{G}{\Fbar LE} F a_0\\
    &= b_n \compose J\lift{K}{E_{n-1}}a_{n-1} \compose \cdots \compose b_1 \compose J\lift{K}{E}a_0\\
    &= J\lift{J}{E_n}b_n \compose J\lift{K}{E_{n-1}}a_{n-1} \compose \cdots \compose J\lift{J}{E_1}b_1 \compose J\lift{K}{E}a_0\\
    &= J\paren[\big]{\lift{J}{E_n}b_n \compose \lift{K}{E_{n-1}}a_{n-1} \compose \cdots \compose \lift{J}{E_1}b_1 \compose \lift{K}{E}a_0}
\end{align*}
Similarly, we see that
\begin{align*}
 J \paren[\big]{\lift{J}{E_m'}b_m' \compose \lift{K}{E_{m-1}'}a_{m-1}'\cdots \compose \lift{J}{E_1'}b_1' \compose \lift{K}{E}a_0'} &= \Fbar d\\
 K \paren[\big]{\lift{J}{E_n}b_n \compose \lift{K}{E_{n-1}}a_{n-1} \compose \cdots \compose \lift{J}{E_1}b_1 \compose \lift{K}{E}a_0 } &= \Gbar d\\
 K \paren[\big]{\lift{J}{E_m'}b_m' \compose \lift{K}{E_{m-1}'}a_{m-1}'\cdots \compose \lift{J}{E_1'}b_1' \compose \lift{K}{E}a_0'} &= \Gbar d,
\end{align*}
and so
\[\lift{J}{E_n}b_n \compose \lift{K}{E_{n-1}}a_{n-1} \compose \cdots \compose \lift{J}{E_1}b_1 \compose \lift{K}{E}a_0 =
\lift{J}{E_m'}b_m' \compose \lift{K}{E_{m-1}'}a_{m-1}'\cdots \compose \lift{J}{E_1'}b_1' \compose \lift{K}{E}a_0'\]
by the independence of \((K,J)\).

We now check that \(L\), with lifts defined in this way, satisfies the lens axioms. The \PutPut{} axiom is immediate from the definition of \(L\), the \PutId{} axiom follows from that of \(K\) (or of \(J\)), and the \PutGet{} axiom holds because
\[L\lift{K}{E}a = \pair{\Gbar L \lift{K}{E}a}{\Fbar L \lift{K}{E}a} = \pair{K\lift{K}{E}a}{J\lift{K}{E}a} = \pair{a}{\lift{G}{JE}Fa} = \pair{a}{\lift{G}{\Fbar LE}Fa} =  \lift{\Gbar}{LE}{a}\]
and similarly \(L \lift{J}{E}b = \lift{\Fbar}{LE}b\) for each \(E \in \objectSet{\E}\), each \(a \in \outSet{\A}{KE}\) and each \(b \in \outSet{\B}{JE}\). The relevant triangles of lenses commute by definition.
\end{proof}

Although the sync minimality of a proxy pullback of a lens cospan is sufficient for the existence of a comparison lens to the proxy-pullback span from an independent lens span that is compatible with the lens cospan, it is not in general necessary. For example, there is always a comparison lens from any proxy-pullback span to itself, namely, the identity lens on its apex. However, sync minimality is in fact necessary for there to be such comparison lenses simultaneously from all of the independent lens spans that are compatible with the lens cospan. 

\begin{theorem}
\label{Sync minimal necessary for terminality}
Consider the proxy-pullback square in \(\Lens\) depicted below.
\[
\begin{tikzcd}
\D \arrow[r, "\Fbar"]\arrow[d, "\Gbar" swap]\arrow[dr, PPB]&
\B \arrow[d, "G"]\\
\A \arrow[r, "F" swap]&
\C
\end{tikzcd}
\]
If the proxy-pullback span \((\Gbar, \Fbar)\) is terminal amongst the independent spans that are compatible with the cospan \((F,G)\), then the proxy-pullback span \((\Gbar, \Fbar)\) is sync minimal.
\end{theorem}

To prove this proposition, we will consider what happens when there is a comparison lens to a proxy pullback from its sync-minimal core.

\begin{lemma}
\label{Sync minimal core and lens structure on inclusion}
Let \(\A \xleftarrow{\Gbar} \D \xrightarrow{\Fbar} \B\) be a lens span. The functor \(\syncCounit{(\Gbar,\Fbar)}\) has a lens structure if and only if it is the identity functor on \(\D\), in which case \((\Gbar,\Fbar) = \sync (\Gbar,\Fbar)\) is sync minimal.
\end{lemma}

\begin{proof}
By construction, the functor \(\syncCounit{(\Gbar,\Fbar)}\) is an identity function on objects and is a subset inclusion on morphisms. A lens that is surjective on objects is also surjective on morphisms~\cite{Clarke:2021:CategoryLens}. Hence, if \(\syncCounit{(\Gbar,\Fbar)}\) is the get functor of a lens, then it is surjective on morphisms and thus actually the identity functor. Conversely, if \(\syncCounit{(\Gbar,\Fbar)}\) is the identity functor on \(\D\), then it is the get functor of the identity lens on~\(\D\).
\end{proof}

\begin{proof}[Proof of \cref{Sync minimal necessary for terminality}.]
Suppose that \((\Gbar, \Fbar)\) is terminal amongst the independent spans that are compatible with \((F,G)\). As the span \((\Gbar, \Fbar)\) is a proxy pullback, it is by definition compatible with the cospan \((F, G)\), and it is independent by \cref{Proxy pullback is independent}. The independence of the span \(\sync (\Gbar, \Fbar)\) and its compatibility with the cospan \((F, G)\) follows from these properties of the span \((\Gbar, \Fbar)\); the former from the way that the sync-minimal core is defined, and the latter because independence is defined in terms of the sync-minimal core. By our assumption, there is thus a comparison lens~\(H\) from the span \(\sync (\Gbar, \Fbar)\) to the span \((\Gbar, \Fbar)\). By the universal property of the pullback span \((\forgetLensToCat \Gbar, \forgetLensToCat \Fbar)\) in \(\Cat\), the functors \(\forget H\) and \(\syncCounit{(\Gbar,\Fbar)}\) are both the unique comparison functor from the span of get functors of \(\sync (\Gbar, \Fbar)\) to the pullback span \((\forgetLensToCat \Gbar, \forgetLensToCat \Fbar)\), and so they are necessarily equal. The result then follows by \cref{Sync minimal core and lens structure on inclusion}.
\end{proof}

\section{Proxy Pullbacks of Split Opfibrations}
\label{Section: proxy pullbacks of split opfibrations}

In the remainder of this paper, we unpack the results in the previous two sections for the proxy pullback of a lens cospan with additional known properties. In this section, we consider what happens when one of the legs of the cospan is a split opfibration. 

\begin{proposition}
\label{Proxy pullback split opfibration terminal}
A proxy-pullback span of a lens cospan with one leg a split opfibration is terminal amongst the independent lens spans that are compatible with the cospan.
\end{proposition}

\Cref{Proxy pullback split opfibration terminal} follows directly from \cref{Sync minimal proxy pullback terminal amongst compatible independent} and the following lemma.

\begin{lemma}
\label{Split opfibration proxy pullback is sync minimal}
Consider a proxy-pullback square
\[
\begin{tikzcd}
\D  \arrow[d, "\Gbar" swap]
    \arrow[r, "\Fbar"]
    \arrow[dr, PPB] &
\B  \arrow[d, "G"]\\
\A  \arrow[r, "F" swap]&
\C
\end{tikzcd}.
\]
If \(F\) or \(G\) is a split opfibration then the lens span \((\Gbar, \Fbar)\) is sync minimal.
\end{lemma}

\begin{proof}
Without loss of generality, suppose that \(F\) is a split opfibration. Let \(d \colon D_1 \to D_2\) be a morphism in \(\D\), and let \(a = \Gbar d \colon A_1 \to A_2\) and \(b = \Fbar d \colon B_1 \to B_2\). Let \(u\) be the unique comparison morphism from the \(F\)-opcartesian morphism \(\lift{F}{A_1}Fa\) to \(a\), as in the diagram
\[\begin{tikzcd}[column sep=huge]
A_1 \arrow[dr, "a" swap]\arrow[r, "\lift{F}{A_1}Fa"]&
A_2'\arrow[d, dashed, "u"]\\
&
A_2
\end{tikzcd}.\]
Then
\(
d = \pair{a}{b} = \pair{u}{\id{B_2}} \compose \pair{\lift{F}{A_1}Fa}{b} = \pair{u}{\lift{G}{B_2}Fu} \compose \pair{\lift{F}{A_1}Gb}{b} = \lift{\Gbar}{\pair{A_2'}{B_2}}u \compose \lift{\Fbar}{\pair{A_1}{B_1}}b
\).
\end{proof}

\begin{remark}
Split opfibrations are pullback stable, so in the proof above \(\Fbar\) is actually a split opfibration. However, lens spans with one leg a split opfibration are not in general sync minimal.
\end{remark}

Shortly we will see that for a lens cospan with one leg a split opfibration, the independence condition for lens spans forming compatible squares with the cospan is equivalent to a simpler notion of independence, which we will call \textit{split independence}.

\begin{definition}
A lens span \(\A \xleftarrow{\Gbar} \D \xrightarrow{\Fbar} \B\) is called \(\Fbar\)-\textit{split independent} if for all \(D_1 \in \objectSet{\D}\), all \(a_1 \colon \Gbar D_1 = A_1 \to A_1'\) in \(\A\), all \(b \colon \Fbar D_1 = B_2 \to B_2\) in \(\B\), and all \(a_2 \colon \target \Gbar \lift{\Fbar}{D_1}b = A_2 \to A_2'\), as shown in the diagram
\begin{equation}
\label{Equation: Simple Independence}
\begin{tikzpicture}[
    x=2.6em, y=2.3em,
    commutative diagrams/.cd,
    every diagram
]
    \begin{scope}
        \node[boxed, fit={(-2,-1.75) (2, 1.75)}, label={below:\(\D\)}] (Dcat) {};
    
        \begin{scope}[
            commutative diagrams/.cd,
            every node,
            every cell
        ]
            \node (D) at (-1,1) {\(D_1\)};
            \node (D1) at (-1,-1) {\(D_1'\)};
            \node (D2) at (1,1) {\(D_2\)};
            \node (D1p) at (0.5,-1) {\(\overCapIt{D_2'}\)};
            \node (D2p) at (1,-0.5) {\(D_2'\)};
        \end{scope}
        
        \path[
            commutative diagrams/.cd,
            every arrow,
            every label
        ]
            (D) edge node[swap] {\(\lift{\Gbar}{D_1}a_1\)} (D1)
                edge node {\(\lift{\Fbar}{D_1}b\)} (D2)
            (D1) edge node[swap] {\(\lift{\Fbar}{D_1'}b\)} (D1p)
            (D2) edge node {\(\lift{\Gbar}{D_2}a_2\)} (D2p)
        ;
    \end{scope}
    
    \begin{scope}[shift={(-4.5,0)}]
        \node[boxed, fit={(-1.5,-1.75) (1.5, 1.75)}, label={below:\(\A\)}] (Acat) {};
        
        \begin{scope}[
            commutative diagrams/.cd,
            every node,
            every cell
        ]
            \node (A1) at (-1,1) {\(A_1\)};
            \node (A1p) at (-1,-1) {\(A_1'\)};
            \node (A2) at (1,1) {\(A_2\)};
            \node (A2p) at (1,-1) {\(A_2'\)};
        \end{scope}
        
        \path[
            commutative diagrams/.cd,
            every arrow,
            every label
        ]
            (A1) edge node {\(\Gbar \lift{\Fbar}{D_1}b\)} (A2)
                edge node[swap] {\(a_1\)} (A1p)
            (A2) edge node {\(a_2\)} (A2p)
            (A1p) edge node[swap] {\(\Gbar \lift{\Fbar}{D_1'}b\)} (A2p)
        ;
    \end{scope}
    
    \begin{scope}[shift={(5.25,0)}]
        \node[boxed, fit={(-2.25,-1.75) (2.25, 1.75)}, label={below:\(\B\)}] (Bcat) {};
    
        \begin{scope}[
            commutative diagrams/.cd,
            every node,
            every cell
        ]
            \node (D) at (-1,1) {\(B_1\)};
            \node (D1) at (-1,-1) {\(B_1\)};
            \node (D2) at (1,1) {\(B_2\)};
            \node (D1p) at (1,-1) {\(B_2\)};
        \end{scope}
        
        \path[
            commutative diagrams/.cd,
            every arrow,
            every label
        ]
            (D) edge[commutative diagrams/equals] node[swap] {\(\Fbar\lift{\Gbar}{D_1}a_1\)} (D1)
                edge node {\(b\)} (D2)
            (D1) edge node[swap] {\(b\)} (D1p)
            (D2) edge[commutative diagrams/equals] node {\(\Fbar\lift{\Gbar}{D_2}a_2\)} (D1p)
        ;
    \end{scope}
    
    \draw[
        commutative diagrams/.cd,
        every arrow,
        every label,
        shorten=0.25em,
        mapsto]
        (Dcat) -- node[swap] {\(\Gbar\)} (Acat);
    
    \draw[
        commutative diagrams/.cd,
        every arrow,
        every label,
        shorten=0.25em,
        mapsto]    
        (Dcat) -- node {\(\Fbar\)} (Bcat);
\end{tikzpicture}\,,
\end{equation}
whenever the square in \(\A\) commutes and \(\Fbar \lift{\Gbar}{D_1}a_1 = \id{B_1}\) and \(\Fbar \lift{\Gbar}{D_2}a_2 = \id{B_2}\), then
also \(D_2' = \overCapIt{D_2'}\) and the resulting square in \(\D\) commutes.
\end{definition}

\begin{proposition}
\label{Independence and split opfibration}
Consider a compatible lens square
\[
\begin{tikzcd}
\D  \arrow[d, "\Gbar" swap]
    \arrow[r, "\Fbar"]&
\B  \arrow[d, "G"]\\
\A  \arrow[r, "F" swap]&
\C
\end{tikzcd}.
\]
If \(F\) is a split opfibration, then \((\Gbar, \Fbar)\) is independent if and only if it is \(\Fbar\)-split independent.
\end{proposition}

The \textit{only if} direction follows directly from the definition of independence. Essential to the proof of the \textit{if} direction is the following lemma.

\begin{lemma}
\label{Split independence and factorisation}
Suppose that \(F\) is a split opfibration and \((\Gbar, \Fbar)\) is \(\Fbar\)-split independent. Then, each morphism \(d \colon D_1 \to D_2\) in the apex of \(\sync (\Gbar, \Fbar)\) has the factorisation
\begin{equation}
    \label{Split opfibration factorisation 1}
    \begin{tikzcd}[column sep=huge]
    D_1 \arrow[r, "\lift{\Fbar}{D_1} \Fbar d"]
        \arrow[dr, "d" swap]&
    D_2' \arrow[d, "\lift{\Gbar}{D_2'} u"]\\
    & D_2
    \end{tikzcd}
\end{equation}
where \(u \colon \Gbar D_3 \to \Gbar D_2\) comes from the universal property of the \(F\)-opcartesian morphism \(\lift{F}{\Gbar D_1}G\Fbar d\), that is, \(u\) is the unique morphism of \(\A\) for which \(Fu = \id{F\Gbar D_2}\) and the diagram
\begin{equation}
     \label{Split opfibration factorisation 2}
    \begin{tikzcd}[column sep=huge]
    \Gbar D_1 \arrow[r, "\lift{F}{\Gbar D_1}G\Fbar d"]
        \arrow[dr, "\Gbar d" swap]&
    \Gbar D_2' \arrow[d, "u"]\\
    & \Gbar D_2
    \end{tikzcd}
\end{equation}
commutes. In particular, the right leg of \(\sync (\Gbar, \Fbar)\) is also a split opfibration.
\end{lemma}

We will return to prove the lemma shortly, but let us first finish the proof of the proposition.

\begin{proof}[Proof of \emph{if} direction of \cref{Independence and split opfibration}.]
If \(F\) is a split opfibration and \((\Gbar, \Fbar)\) is \(\Fbar\)-split independent, then \cref{Split independence and factorisation} implies that each morphism \(d\) in the apex of \(\sync (\Gbar, \Fbar)\) is uniquely determined by the data \(\Gbar d\) and \(\Fbar d\). Indeed, from \cref{Split opfibration factorisation 1}, \(d\) is a composite of morphisms expressed in terms of \(\Fbar d\) and \(u\), and \(u\) itself is uniquely determined by the top side and hypotenuse of the triangle \cref{Split opfibration factorisation 2}, which are themselves expressed in terms of \(\Gbar d\) and \(\Fbar d\).
\end{proof}

\begin{proof}[Proof of \cref{Split independence and factorisation}.]
Recall that each morphism in the apex of \(\sync{(\Gbar, \Fbar)}\) is a composite 
\[
\begin{tikzcd}[column sep=large]
D_1 \arrow[r, "d_1"]&
D_2 \arrow[r, "d_2"]&
D_3 \arrow[r, phantom, "\cdots"]&[-2em]
D_{n-1} \arrow[r, "d_{n-1}"]&
D_{n}
\end{tikzcd}
\]
of morphisms in \(\D\) such that, for each \(k\), either \(d_k = \lift{\Gbar}{D_k} a_k\) or \(d_k = \lift{\Fbar}{D_k}b_k\), where \(a_k = \Gbar d_k\) and \(b_k = \Fbar d_k\). We will inductively construct the dashed morphisms in the diagram
\[
\begin{tikzcd}[column sep=huge]
D_1' \arrow[r, "\lift{\Fbar}{D_1'}b_1", dashed]
    \arrow[d, "\lift{\Gbar}{D_1'}u_1" swap, equals]&
D_2' \arrow[r, "\lift{\Fbar}{D_2'}b_2", dashed]
    \arrow[d, "\lift{\Gbar}{D_2'}u_2" swap, dashed]&
D_3'\arrow[r, phantom, "\cdots"]
    \arrow[d, "\lift{\Gbar}{D_3'}u_3" swap, dashed]&[-1em]
D_{n-1}' \arrow[r, "\lift{\Fbar}{D_{n-1}'}b_{n-1}", dashed]
    \arrow[d, "\lift{\Gbar}{D_{n-1}'}u_{n-1}", dashed]&
D_{n}' \arrow[d, "\lift{\Gbar}{D_n'}u_n", dashed]\\
D_1 \arrow[r, "d_1" swap]&
D_2 \arrow[r, "d_2" swap]&
D_3 \arrow[r, phantom, "\cdots"]&[-2em]
D_{n-1} \arrow[r, "d_{n-1}" swap]&
D_{n}
\end{tikzcd}
\]
in \(\D\) such that the resulting diagram in \(\D\) commutes and \(\Fbar \lift{\Gbar}{D'_k}u_k = \id{\Fbar D_k}\) for each \(k\). 

For the base step, we may set \(D_1' = D_1\) and \(u_1 = \id{\Gbar D_1}\), so that \(\Fbar\lift{\Gbar}{D_1'}u_1 = \Fbar\id{D_1} = \id{\Fbar D_1}\).

For the inductive step, suppose that we have already constructed \(D_k'\) and \(u_k\), and wish now to construct \(D_{k + 1}'\) and \(u_{k + 1}\). Consider the diagram
\[
\begin{tikzpicture}[
    x=3.6em, y=2.4em,
    commutative diagrams/.cd,
    every diagram
]
    \begin{scope}[shift={(-2.5,2.25)}]
        \node[boxed, fit={(-1.8,-1.6) (2.2, 1.6)}, label={left:\(\D\)}] (Dcat) {};
    
        \begin{scope}[
            commutative diagrams/.cd,
            every node,
            every cell
        ]
            \node (D) at (-1,1) {\(D'_k\)};
            \node (D1) at (-1,-1) {\(D_k\)};
            \node (D2) at (1,1) {\(D_{k + 1}'\)};
            \node (D1p) at (0.5,-1) {\(D_{k + 1}\)};
            \node (D2p) at (1,-0.5) {\(\overCapIt{D_{k + 1}}\)};
        \end{scope}
        
        \path[
            commutative diagrams/.cd,
            every arrow,
            every label
        ]
            (D) edge node[swap] {\(\lift{\Gbar}{D_k'}u_k\)} (D1)
                edge[dashed] node {\(\lift{\Fbar}{D_k'}b_k\)} (D2)
            (D1) edge node[swap] {\(d_k\)} (D1p)
            (D2) edge[dashed] node {\(\lift{\Gbar}{D'_{k + 1}} u_{k + 1}\)} (D2p)
        ;
    \end{scope}
    
    \begin{scope}[shift={(2.5,-2.25)}]
        \node[boxed, fit={(-1.6,-1.6) (1.6, 1.6)}, label={right:\(\C\)}] (Ccat) {};
        
        \begin{scope}[
            commutative diagrams/.cd,
            every node,
            every cell
        ]
            \node (D) at (-1,1) {\(C_k\)};
            \node (D1) at (-1,-1) {\(C_k\)};
            \node (D2) at (1,1) {\(C_{k + 1}\)};
            \node (D1p) at (1,-1) {\(C_{k + 1}\)};
        \end{scope}
        
        \path[
            commutative diagrams/.cd,
            every arrow,
            every label
        ]
            (D) edge[commutative diagrams/equals] (D1)
                edge node {\(c_k\)} (D2)
            (D1) edge node[swap] {\(c_k\)} (D1p)
            (D2) edge[commutative diagrams/equals] (D1p)
        ;
    \end{scope}
    
    \begin{scope}[shift={(-2.5,-2.25)}]
        \node[boxed, fit={(-1.8,-1.6) (2.2, 1.6)}, label={left:\(\A\)}] (Acat) {};
        
        \begin{scope}[
            commutative diagrams/.cd,
            every node,
            every cell
        ]
            \node (A1) at (-1,1) {\(A_k'\)};
            \node (A1p) at (-1,-1) {\(A_k\)};
            \node (A2) at (1,1) {\(A_{k + 1}'\)};
            \node (A2p) at (1,-1) {\(A_{k + 1}\)};
        \end{scope}
        
        \path[
            commutative diagrams/.cd,
            every arrow,
            every label
        ]
            (A1) edge[dashed] node {\(\lift{F}{A'_k}c_k\)} (A2)
                edge node[swap] {\(u_k\)} (A1p)
            (A2) edge[dashed] node {\(u_{k + 1}\)} (A2p)
            (A1p) edge node[swap] {\(a_k\)} (A2p)
        ;
    \end{scope}
    
    \begin{scope}[shift={(2.5,2.25)}]
        \node[boxed, fit={(-1.6,-1.6) (1.6, 1.6)}, label={right:\(\B\)}] (Bcat) {};
    
        \begin{scope}[
            commutative diagrams/.cd,
            every node,
            every cell
        ]
            \node (D) at (-1,1) {\(B_k\)};
            \node (D1) at (-1,-1) {\(B_k\)};
            \node (D2) at (1,1) {\(B_{k + 1}\)};
            \node (D1p) at (1,-1) {\(B_{k + 1}\)};
        \end{scope}
        
        \path[
            commutative diagrams/.cd,
            every arrow,
            every label
        ]
            (D) edge[commutative diagrams/equals] (D1)
                edge node {\(b_k\)} (D2)
            (D1) edge node[swap] {\(b_k\)} (D1p)
            (D2) edge[commutative diagrams/equals] (D1p)
        ;
    \end{scope}
    
    \draw[
        commutative diagrams/.cd,
        every arrow,
        every label,
        shorten=0.2em,
        mapsto]
        (Dcat) -- node[swap] {\(\Gbar\)} (Acat);
    
    \draw[
        commutative diagrams/.cd,
        every arrow,
        every label,
        shorten=0.2em,
        mapsto]    
        (Dcat) -- node {\(\Fbar\)} (Bcat);
        
    \draw[
        commutative diagrams/.cd,
        every arrow,
        every label,
        shorten=0.2em,
        mapsto]
        (Bcat) -- node[] {\(G\)} (Ccat);
    
    \draw[
        commutative diagrams/.cd,
        every arrow,
        every label,
        shorten=0.2em,
        mapsto]    
        (Acat) -- node[swap] {\(F\)} (Ccat);
\end{tikzpicture}
\]
From the universal property of the \(F\)-opcartesian morphism \(\lift{F}{A'_k}c_k\), there is a unique morphism \(u_{k + 1} \colon A'_{k + 1} \to A_{k + 1}\) in \(\A\) above \(\id{C_{k + 1}}\) such that the square in \(\A\) above commutes.

Suppose that \(d_k = \lift{\Gbar}{D_k} a_k\). By the \PutPut{} axiom, and commutativity of the lens square,
\begin{equation}
    \label{Equation: case dk lift from A}
    d_k \compose \lift{\Gbar}{D_k'}u_k = \lift{\Gbar}{D'_{k + 1}} u_{k + 1} \compose \lift{\Gbar}{D_k'} \lift{F}{A_k'}c_k = \lift{\Gbar}{D'_{k + 1}} u_{k + 1} \compose \lift{\Fbar}{D_k'} \lift{G}{B_k}c_k.
\end{equation}
We also have \(\Fbar \lift{\Gbar}{D'_k} u_k = \lift{G}{B_k} F u_k = \lift{G}{B_k}\id{C_k} = \id{B_k}\) by compatibility of the lens square, and similarly \(\Fbar \lift{\Gbar}{D'_{k + 1}}u_{k + 1} = \id{B_{k + 1}}\). Hence, applying \(\Fbar\) to both sides of \cref{Equation: case dk lift from A}, we see that \(b_k = \lift{G}{B_k}c_k\). Thus \cref{Equation: case dk lift from A} actually says that \(\overCapIt{D_{k + 1}} = D_{k + 1}\) and the square in \(\D\) above commutes.

Otherwise, \(d_k = \lift{\Fbar}{D_k} b_k\), and thus also \(a_k = \Gbar \lift{\Fbar}{D_k}b_k\). Also, as the lens square is compatible, \(\lift{F}{A_k'}c_k = \Gbar \lift{\Fbar}{D_k'}{b_k}\). As \((\Gbar, \Fbar)\) is split independent, again \(\overCapIt{D_{k + 1}} = D_{k + 1}\) and the square in \(\D\) above commutes.
\end{proof}

\section{Proxy Pullbacks of Discrete Opfibrations and Proxy Products}
\label{Section: proxy pullbacks of discrete opfibrations and proxy products}

The results in the previous section apply in particular to proxy pullbacks of lens cospans with one leg a discrete opfibration. For proxy pullbacks of such cospans, \cref{Proxy pullback split opfibration terminal} simplifies as follows.

\begin{proposition}
\label{Proxy pullback of discrete opfibration is pullback}
Proxy pullbacks of discrete opfibrations are real pullbacks in \(\Lens\).
\end{proposition}

This result was actually first proved by Chollet et al.~\cite{Clarke:2021:CategoryLens}, but in a nuts-and-bolts manner rather than as a consequence of the general theory that we present in this paper.

\begin{lemma}
\label{Independence and discrete opfibration}
Consider a compatible lens square
\[
\begin{tikzcd}
\D  \arrow[d, "\Gbar" swap]
    \arrow[r, "\Fbar"]&
\B  \arrow[d, "G"]\\
\A  \arrow[r, "F" swap]&
\C
\end{tikzcd}.
\]
If \(F\) or \(G\) is a discrete opfibration, then \((\Gbar, \Fbar)\) is independent.
\end{lemma}

\begin{proof}
Without loss of generality, suppose that \(F\) is a discrete opfibration. Let \(D_1 \in \objectSet{\D}\), let \(a_1 \colon \Gbar D_1 = A_1 \to A_1'\) in \(\A\), let \(b \colon \Fbar D_1 = B_2 \to B_2\) in \(\B\), and let \(a_2 \colon \target \Gbar \lift{\Fbar}{D_1}b = A_2 \to A_2'\), as shown in the diagram \cref{Equation: Simple Independence}. Suppose also that the square in \(\A\) commutes, and that \(\Fbar \lift{\Gbar}{D_1}a_1 = \id{B_1}\) and \(\Fbar \lift{\Gbar}{D_2}a_2 = \id{B_2}\). We have
\[Fa_1 = G\lift{G}{B_1}Fa_1 = G \Fbar \lift{\Gbar}{D_1} a_1 = G\id{B_1} = \id{GB_1} = \id{FA_1}\]
by compatibility of the lens square, and so \(a_1 = \id{A_1}\) as \(F\) is a discrete opfibration. Hence \(\lift{\Gbar}{D_1}{a_1} = \lift{\Gbar}{D_1}{\id{A_1}} = \id{D_1}\) and \(D_1' = D_1\). Similarly, \(\lift{\Gbar}{D_2}a_2 = \id{D_2}\) and \(D_2' = D_2\). As \(D_1' = D_1\), we have \(\lift{\Fbar}{D_1}b = \lift{\Fbar}{D_1'} b\). Hence the square in \(\D\) commutes.
\end{proof}

\begin{proof}[Proof of \cref{Proxy pullback of discrete opfibration is pullback}.]
By \cref{Proxy pullback split opfibration terminal}, a proxy-pullback span of a lens cospan with one leg a discrete opfibration is terminal amongst the independent lens spans that are compatible with the cospan. Every lens span that forms a commuting square with such a cospan is actually compatible with the cospan by \cref{Compatible square and discrete opfibration} and independent by \cref{Independence and discrete opfibration}.
\end{proof}

Specialising further, we now consider the proxy pullbacks of those cospans whose apex is the terminal category, that is, proxy products. Recall that the unique lens from a category \(\C\) to the terminal category is a discrete opfibration if and only if \(\C\) is a discrete category. The specialisation of \cref{Proxy pullback of discrete opfibration is pullback} then says that the proxy product of a category with a discrete category is a real product. Actually, in this case, the converse also holds.

\begin{proposition}
\label{Proxy product of discrete categories}
The proxy product of two categories is a real product if and only if at least one of the two categories is a discrete category.
\end{proposition}

To prove the converse, it suffices to show, for all non-discrete categories \(\A\) and \(\B\), that there is a non-independent lens span from \(\A\) to \(\B\). We may explicitly describe such a non-independent lens span; it is merely the so-called \textit{funny tensor product} \(\A \freeProduct \B\) of \(\A\) and \(\B\) with a canonical lens structure on the projection functors. Henceforth, we will refer to the funny tensor product as the \textit{free product} of categories, as it generalises the well-known free product of groups. The free product of categories has several different descriptions; we will use the following one.

\begin{definition}
The \textit{free product} \(\A \freeProduct \B\) of categories \(\A\) and \(\B\) is the category with object set \(\objectSet{\A} \times \objectSet{\B}\) whose morphisms are freely generated by those of the form
\[(A_1,B) \xrightarrow{(a, B)} (A_2, B) \qquad\text{and}\qquad (A,B_1) \xrightarrow{(A, b)} (A, B_2),\]
subject to the equations
\begin{align*}
    (\id{A}, B) &= \id{(A,B)} &
    (a_2, B) \compose (a_1, B) &= (a_2 \compose a_1, B)\\
    (A, \id{B}) &= \id{(A,B)} &
    (A, b_2) \compose (A, b_1) &= (A, b_2 \compose b_1).
\end{align*}
There are projection lenses \(P_1 \colon \A \freeProduct \B \to \A\) and \(P_2 \colon \A \freeProduct \B \to \B\), defined by the equations
\begin{align*}
    P_1 (A, B) &= A & P_1 (a, B) &= a & P_1 (A, b) &= \id{A} & \lift{P_1}{(A, B)}a = (a, B)\\
    P_2 (A, B) &= B & P_2 (A, b) &= b & P_2 (a, B) &= \id{B} & \lift{P_2}{(A, B)}b = (A, b),
\end{align*}
whose get functors are the usual projection functors.
\end{definition}

\begin{proof}[Proof of \cref{Proxy product of discrete categories}.]
The \textit{if} direction is a particular case of \cref{Proxy pullback of discrete opfibration is pullback}. For the \textit{only if} direction, suppose that \(\A\) and \(\B\) are both non-discrete categories, that is, that there are non-identity morphisms \(a \colon A_1 \to A_2\) in \(\A\) and \(b \colon B_1 \to B_2\) in \(\B\). Then the morphisms
\[(A_1, B_1) \xrightarrow{(a, B_1)} (A_2, B_1) \xrightarrow{(A_2, b)} (A_2, B_2) \qquad \text{and} \qquad (A_1, B_1) \xrightarrow{(A_1, b)} (A_1, B_2) \xrightarrow{(a, B_2)} (A_2, B_2)\]
in \(\A \freeProduct \B\) both have the same source object, and both are mapped by \(P_1\) to \(a\) and \(P_2\) to \(b\), but they are not equal. Hence the lens span \(\A \xleftarrow{P_1} \A \freeProduct \B \xrightarrow{P_2} \B\) is not independent.
\end{proof}

\section{Conclusion}
\label{Chapter: Conclusion}

In this paper, we gave necessary and sufficient conditions for when a lens span that forms a commuting square with a lens cospan has a comparison lens to a proxy pullback of the cospan. These conditions involved the new notions of compatibility, sync minimality and independence. They enabled us to describe exactly when a proxy pullback is a real pullback, and this description simplified further for proxy products. A search for such a simplified description for general proxy pullbacks is ongoing.

We would like to obtain categorical characterisations of the notions of sync minimality and independence, perhaps in terms of some universal property. Whilst the author is yet to discover a compelling such characterisation of independent lens spans, some interesting progress has already been made for sync minimal ones. The key observation is that being sync minimal is really a property of the put cofunctors of a lens span, and that the process of taking the sync minimal core actually gives a factorisation of this span of put cofunctors. Spivak and Niu~\cite{spivak:2021:polynomial-functors} show that \(\Cof\) has products; the explicit description of these products is unfortunately rather complicated—the objects of the product of two categories are certain pairs of rooted infinite trees whose edges are morphisms from either category, and the morphisms out of such an object are the paths in either tree from its root. It turns out that a cofunctor span is sync minimal exactly when its product pairing in \(\Cof\) has surjective lifting functions. We will call a cofunctor with surjective lifting functions \textit{cofull} and one with injective lifting functions \textit{cofaithful}. There is a well-known factorisation system on \(\Cof\) whose left class is the bijective-on-objects cofunctors and whose right class is the discrete opfibrations~\cite{Clarke:2020:InternalLensesAsFunctorsAndCofunctors}, which, in this context, we might also call the cofull cofaithful cofunctors. The factorisation system on \(\Cof\) that we are actually interested in has as its left class the cofaithful bijective-on-objects cofunctors, and its right class the cofull cofunctors; this factorisation of the put cofunctor of a lens coincides with the other factorisation. If we factor the product pairing of a cofunctor span using this factorisation system, then the sync-minimal core of the cospan is obtained by composing the second factor with the appropriate product projection cofunctors.

We have already recalled that symmetric lenses between two categories correspond to the equivalence classes of a certain equivalence relation on asymmetric lens spans between the two categories~\cite{JohnsonRosebrugh:2015:SpansDeltaLenses}. Clarke, with a different definition of symmetric lens, constructed an adjoint triple\footnote{Clarke's functor \(\mathscr{M}\) is not to be confused with our \(\mathscr{M}\) that sends a lens span to its sync-minimal core.}
\[
\begin{tikzcd}[column sep=large]
\SymLens(\A,\B)
    \arrow[r, "\mathscr{L}", shift left=5]
    \arrow[r, upLeftAdjointTo, shift left=3]
    \arrow[r, upLeftAdjointTo, shift right=3]
    \arrow[r, "\mathscr{R}" swap, shift right=5]&
\SpanLens(\A,\B)
    \arrow[l, "\mathscr{M}" description]
\end{tikzcd}
\]
between his category \(\SymLens(\A,\B)\) of symmetric lenses from \(\A\) to \(\B\) and the category \(\SpanLens(\A,\B)\) whose objects are lens spans from \(\A\) to \(\B\) and whose morphisms are functors satisfying certain compatibility conditions~\cite{Clarke:2021:ADiagrammaticApproachToSymmetricLenses}. The comonad \(\mathscr{L} \compose \mathscr{M}\) on \(\SpanLens(\A,\B)\) induced by the adjoint triple appears to be closely related to our process that sends a lens span to its sync minimal core. Additionally, as \(\mathscr{L}\) is fully faithful, we may think of those lens spans in the image of \(\mathscr{L}\) as representing symmetric lenses. It might thus be reasonable to think of the sync-minimal lens spans as being the symmetric lenses, an idea that is reinforced by the interpretation of the sync-minimal property that was given in \cref{Section: Sync-minimal spans}.

The original proposal for the \textit{Categories of Maintainable Relations} project of the Applied Category Theory Adjoint School 2020, which did not end up being the actual focus of the project, was to work out how to view symmetric lenses as some kind of generalised relations in~\(\Lens\). A \textit{relation} in a category from object \(X\) to object \(Y\) is usually defined as a jointly monic span from \(X\) to \(Y\). A \textit{regular category}~\cite{barrgrillet:1971:exactcategories} is a finitely complete category with a pullback-stable regular-epi mono factorisation system. Relations in regular categories are particularly nice as they form the morphisms of a bicategory; the composite of two relations is the image (from the factorisation system) of their composite as spans (from the pullback). Given a not-necessarily-proper orthogonal factorisation system \((\curlyscr{E}, \curlyscr{M})\) on a category with products, an \(\curlyscr{M}\)-relation from \(X\) to \(Y\) is a span from \(X\) to \(Y\) whose product pairing is in \(\curlyscr{M}\). If the factorisation system is pullback-stable, then the \(\curlyscr{M}\)-relations still form the morphisms of a bicategory with nice properties~\cite{meisen:1974:bicategories-relations-and-pullback,kelly:1991:note-relations-relative-factorization,pavlovic:1995:maps-relative-factorisation-system}, where composition of \(\curlyscr{M}\)-relations is defined similarly to that of relations in a regular category. As \(\Cof\) is finitely complete, we may consider the \(\curlyscr{M}\)-relations in \(\Cof\) for the factorisation system where \(\curlyscr{E}\) is the class of cofaithful bijective-on-objects cofunctors and \(\curlyscr{M}\) is the class of cofull cofunctors. From our earlier discussion, these \(\curlyscr{M}\)-relations are exactly the sync-minimal cofunctor spans. It would be interesting to work out what the composition of such \(\curlyscr{M}\)-relations is, as the pullback in \(\Cof\) is very different to the proxy pullback in \(\Lens\). Returning to the question of whether symmetric lenses may be viewed as some kind of relations in~\(\Lens\), we seem to need a further generalisation of the notion of internal relation as the sync-minimal core of a lens span is not obtained from a factorisation system on \(\Lens\) itself.

\section*{Acknowlegdements}

Many thanks go to Michael Johnson, Bryce Clarke, Richard Garner, Samuel Muller, Chris Heunen, Paolo Perrone and Perdita Stevens for their helpful feedback and suggestions.

\raggedright
\bibliographystyle{EPTCS/eptcs}
\bibliography{references}
\end{document}